\documentclass[11pt]{amsart}
\usepackage{amsmath, amssymb, amsthm, amsfonts, mathrsfs}
\usepackage[all,cmtip]{xy}

\newtheorem{theorem}{Theorem}[section]
\newtheorem{proposition}[theorem]{Proposition}
\newtheorem{corollary}[theorem]{Corollary}
\newtheorem{lemma}[theorem]{Lemma}

\theoremstyle{definition}
\newtheorem{definition}[theorem]{Definition}

\theoremstyle{definition}

\newcommand{\MPT}{{\bf{MPT}}}
\newcommand{\Rone}{{\bf R}_1}

\theoremstyle{definition}

\DeclareSymbolFont{AMSb}{U}{msb}{m}{n}
\DeclareMathSymbol{\N}{\mathbin}{AMSb}{"4E}
\DeclareMathSymbol{\Z}{\mathbin}{AMSb}{"5A}
\DeclareMathSymbol{\R}{\mathbin}{AMSb}{"52}
\DeclareMathSymbol{\Q}{\mathbin}{AMSb}{"51}
\DeclareMathSymbol{\I}{\mathbin}{AMSb}{"49}
\DeclareMathSymbol{\C}{\mathbin}{AMSb}{"43}

\begin{document}
\title{A Model for Rank One Measure Preserving Transformations}
\author{Su Gao}
\address{Department of Mathematics\\ University of North Texas\\ 1155 Union Circle \#311430\\  Denton, TX 76203\\ USA}
\email{sgao@unt.edu}
\author{Aaron Hill}
\address{Department of Mathematics\\ University of North Texas\\ 1155 Union Circle \#311430\\  Denton, TX 76203\\ USA}
\email{Aaron.Hill@unt.edu}
\date{\today}
\thanks{The first author acknowledges the US NSF grants DMS-0901853 and DMS-1201290 for the support of his research. The second author acknowledges the US NSF grant DMS-0943870 for the support of his research.}
\maketitle \thispagestyle{empty}

\begin{abstract}
We define a model for rank one measure preserving transformations in the sense of \cite{Foreman}. This is done by defining a new Polish topology on the space of codes, which are infinite rank one words, for symbolic rank one systems. We establish that this space of codes has the same generic dynamical properties as the space of (rank one) measure preserving transformations on the unit interval.
\end{abstract}


\section{Introduction}
There are several definitions of rank one measure preserving transformations in the literature (cf. a summary in \cite{Ferenczi}). They are generally considered equivalent, and none are easy to give. Among them, the two most useful have been the constructive geometric definition and the constructive symbolic definition. According to the constructive geometric definition, a rank one transformation is a measure preserving transformation of the unit interval that is obtained by a cutting and stacking process. The constructive symbolic definition, however, defines a rank one system as a special kind of Bernoulli subshift.

In ergodic theory it is important to speak of a generic dynamical property of measure preserving transformations. In order to do this one needs to fix a topology on the space of all measure preserving transformations. Let $\MPT$ denote the collection of all invertible measure preserving transformations on the unit interval with the Lebegue measure (by identifying transformations that agree on a set of measure one). Then the weak topology on $\MPT$ is Polish and is considered the standard topology.

Let $\Rone$ be the subcollection of $\MPT$ consisting of all rank one transformations, i.e., transformations that are isomorphic to constructive rank one transformations. Then $\Rone$ is a dense $G_\delta$ subset of $\MPT$. Thus  $\Rone$ is not only a Polish space in its own right with the subspace topology inherited from $\MPT$, but also a generic class in $\MPT$. In particular, a dynamical property of rank one transformations is generic iff it is generic as a property for all measure preserving transformations.

The situation on the side of the constructive symbolic definition of rank one systems, however, is not clear. Recently the authors defined in \cite{GaoHill1} a natural space $\mathcal{R}$ of codes for all symbolic rank one homeomorphisms. We also showed that it has a natural Polish topology. Thus it makes sense to consider the subspace $\mathcal{R}^*$ of all codes for symbolic systems that correspond to rank one measure preserving systems. Unfortunately $\mathcal{R}^*$ is neither generic in $\mathcal{R}$ nor Polish with the subspace topology. Thus in order to speak of generic dynamical properties on $\mathcal{R}^*$ we need to redefine a Polish topology on $\mathcal{R}^*$.

As $\mathcal{R}^*$ is a Borel subset of $\mathcal{R}$ such Polish topology on $\mathcal{R}^*$ definitely exists. However, we would like the topology on $\mathcal{R}^*$ to have better properties. First, it is desirable that the topology on $\mathcal{R}^*$ is naturally connected to its meaning as the space of all rank one measure preserving systems. Second, it would be nice if with an appropriate definition of the topology on $\mathcal{R}^*$ the two spaces $\mathcal{R}^*$ and $\Rone$ have the same generic dynamical properties. Moreover, we want an explicit correspondence between $\mathcal{R}^*$ and $\Rone$ that would map a code in $\mathcal{R}^*$ to a generic transformation in $\Rone$ isomorphic to the coded system.

We achieve all these in this paper. We will define a natural Polish topology on $\mathcal{R}^*$ and prove that $\mathcal{R}^*$ and $\Rone$ share the same generic dynamical properties. This is done by making use of the concept of a {\it model} defined by Foreman, Rudolph and Weiss in \cite{Foreman}. The main theorem of the current paper is to show that $\mathcal{R}^*$ is a model of (rank one) measure preserving transformations in the sense of \cite{Foreman}. That $\mathcal{R}^*$ and $\Rone$ share the same generic dynamical properties is a corollary of the main theorem.

Since many interesting results about rank one transformations were proved by combinatorial methods applied to the symbolic context, the results established in this paper are potentially useful for further studies of generic behavior of rank one transformations.

\section{The standard model}

In this section we recall the basic definitions and establish the standard model for rank one measure preserving transformations.

A {\em measure preserving system} is a quadruple $(X, \mathcal{B}, \mu, T)$ where $X$ is a set, $\mathcal{B}$ is a $\sigma$-algebra of subsets of $X$, $\mu$ a separable non-atomic probability measure on $\mathcal{B}$, and $T$ an invertible $\mu$-preserving transformation, i.e., $T: X\to X$ is an invertible map such that for all $A\in \mathcal{B}$, $T^{-1}(A)\in\mathcal{B}$ and $\mu(T^{-1}(A))=\mu(A)$.

In the above definition the triple $(X,\mathcal{B}, \mu)$ is called a {\em Lebegue space}. It is well known that any Lebesgue space is isomorphic, modulo a null set, to the unit interval $[0,1]$ with the $\sigma$-algebra of all Borel sets and the standard Lebesgue measure $\lambda$. For notational simplicity we denote this canonical Lebesgue space by just $[0,1]$. It follows that any measure preserving system is isomorphic to one on $[0,1]$ with some Lebesgue measure preserving $T:[0,1]\to[0,1]$. It is thus natural to consider the collection
$$ \MPT=\{T:[0,1]\to [0,1]\,\mid\, \mbox{$T$ is Lebesgue measure preserving}\}, $$
with identification of $T$ and $S$ if $$\lambda(\{x\in[0,1]\,\mid\, T(x)=S(x)\})=1,$$
 a {\em model} of all measure preserving systems. (Here the word ``model" is used in the intuitive sense; later on we will use it in a rigorous sense following a definition of Foreman, Rudolph and Weiss \cite{Foreman}.) The standard topology on $\MPT$ is the {\em weak topology}, i.e., the topology generated by the convergence
$$ T_n\to T \mbox{ iff } \lambda(T_n(A)\triangle T(A))\to 0 \mbox{ for all Borel $A\subseteq [0,1]$.} $$
It is well known that this topology is Polish.

If $(X, \mathcal{B}, \mu, T)$ and $(Y, \mathcal{C},\nu, S)$ are two measure preserving systems, a ({\em measure-theoretic}) {\em isomorphism} between them is a bijection $\phi: X'\to Y'$, where $X'\subseteq X$ and $Y'\subseteq Y$ with $\mu(X')=\nu(Y')=1$, such that $\phi$ is bimeasurable,
$\mu(\phi^{-1}(A))=\nu(A)$ for all $A\in \mathcal{C}$, and $S\circ \phi=\phi\circ T$. When there is such an isomorphism, we write $(X,\mathcal{B},\mu, T)\cong (Y,\mathcal{C},\nu, S)$, or simply $(X, T)\cong (Y,S)$ or $T\cong S$ if the Lebesgue spaces are defined without ambiguity.

Following \cite{GlasnerKing} and \cite{Foreman} we define a {\em dynamical property} of measure preserving transformations to be an isomorphism-invariant class of measure preserving transformations. Formally, a dynamical property is a class $\Gamma$ of measure preserving transformations such that if $T\in \Gamma$ and $S\cong T$, then $S\in \Gamma$. Dynamical properties can be modeled by $\cong$-invariant subsets of $\MPT$. A dynamical property $\Gamma$ is {\em generic} if $\Gamma\cap\MPT$ is a comeager (or residual) subset of $\MPT$. This makes sense because we equipped $\MPT$ with a (standard) topology as above.

In this paper we study the space of all rank one measure preserving transformations. Among some equivalent definitions of rank one we will use the {\em constructive geometric definition} as the standard one. This is given below.

\begin{definition}\label{defgeo} A measure preserving transformation $T$ on $[0,1]$ is {\em rank one} if there exist sequences of positive integers $q_n\geq 2$, $n \in\mathbb{N}$, and nonnegative integers $a_{n,i}$, $n \in \mathbb{N}$, $1 \leq i\leq q_n-1$, such that, if $h_n$ is defined by
$$h_0 = 1, h_{n+1} = q_nh_n + \sum^{q_n-1}_{i=1} a_{n,i},$$
then
$$ \sum^{+\infty}_{n=0}\displaystyle\frac{h_{n+1}-q_n h_n}{h_{n+1}}<+\infty,$$
and subsets of $[0,1]$, denoted by $F_n$, $n\in\mathbb{N}$, by $F_{n,i}$, $n\in\mathbb{N}$, $1\leq i\leq q_n$, and by $C_{n,i,j}$, $n\in\mathbb{N}$, $1\leq i\leq q_n-1$, $1\leq j \leq a_{n,i}$, (if $a_{n,i} = 0$ then there are no $C_{n,i,j}$), such that for all $n$
\begin{itemize}
\item $\{F_{n,i}\mid 1\leq i\leq q_n\}$ is a partition of $F_n$,
\item the $T^k(F_n)$, $0\leq k\leq h_n-1$, are disjoint, 
\item $T(F_{n, i}) = C_{n,i,1}$ if $a_{n,i}\neq 0$ and $i < q_n$,
\item $T(F_{n,i}) = F_{n,i+1}$ if $a_{n,i} = 0$ and $i < q_n$,
\item $T(C_{n,i,j}) = C_{n,i,j+1}$ if $j < a_{n,i}$,
\item $T(C_{n,i,a_{n,i}}) = F_{n,i+1}$ if $i < q_n$,
\item $F_{n+1} = F_{n,1}$,
\end{itemize}
and $\lambda(\bigcup^{h_n-1}_{k=0} T^k(F_n))\to 1$ as $n\to \infty$.
\end{definition}

What the definition is trying to describe is a Rokhlin tower obtained by a {\em cutting and stacking} process. In this definition the sequence $(q_n)$ is called the {\em cutting parameter}, the sets $C_{n,i,j}$ are called the {\em spacers}, and the doubly-indexed sequence $(a_{n,i})$ is called the {\em spacer parameter}.

When the spacer parameter is the constant zero sequence, i.e., when the spacers are nonexistent, the transformation is called an {\em odometer map}.

A general measure preserving transformation (on a general Lebesgue space) is {\em rank one} if it is isomorphic to a constructive rank one transformation on $[0,1]$ as defined above. With this definition, it is trivial that being rank one is a dynamical property. Let $\Rone$ be the collection of all rank one transformations in $\MPT$. Then it is well known (but somewhat difficult to find in the literature) that $\Rone$ is a dense $G_\delta$ subspace of $\MPT$. The density  follows, in particular, from the following corollary of Rokhlin's lemma: if $T$ is an aperiodic measure preserving transformation then the collection of all measure preserving transformations isomorphic to $T$ is dense in $\MPT$. The fact that $\Rone$ is $G_\delta$ can be seen from an equivalent definition of rank one which Ferenczi called the reduced geometric definition (\cite{Ferenczi}).

Thus being rank one is itself a generic dynamical property of all measure preserving transformations.

The main technical theorem of this paper concerns the following concept defined in \cite{Foreman}.

\begin{definition}\label{defmodel} A {\em model} for measure preserving transformations is a pair $(X,\pi)$, where $X$ is a Polish space and $\pi: X\to \MPT$ is a continuous function such that for a comeager set $C\subseteq\MPT$ and for all $T\in C$, $\{x\in X\,\mid\, \pi(x)\cong T\}$ is dense in $X$.
\end{definition}

In practice, the space $X$ in a model is a space of measure preserving systems or ``codes" for measure preserving systems. When this is the case, we require the map $\pi$ to be ``isomorphism-preserving," i.e. $\pi(x)\cong x$ for all $x\in X$. The point of this concept rests upon an immediate corollary (Lemma 12 of \cite{Foreman}) that if $X$ is a model for measure preserving transformations, then $X$ and $\MPT$ have the same generic dynamical properties. This result was stated without proof in \cite{Foreman}. For the convenience of the reader we give a proof before the end of this section.

With this rigorous definition $\Rone$ becomes a model with the identity map into $\MPT$. It follows that $\Rone$ and $\MPT$ have the same generic dynamical properties. Henceforth we will call $\Rone$ the {\em standard model} for rank one measure preserving transformations.

\begin{lemma}[\cite{Foreman} Lemma 12]\label{lemma12} If $(X, \pi)$ is a model for measure preserving transformations, then $X$ and $\MPT$ have the same generic dynamical properties.
\end{lemma}
\begin{proof} First assume $P\subseteq \MPT$ is a generic dynamical property, i.e., $P$ is an invariant comeager subset of $\MPT$. Let $P'\subseteq P$ be a dense $G_\delta$. Note that $\MPT$ is a Polish group acting on itself by conjugacy, and that the action is continuous. Thus we have that
$$ \forall g\in \MPT\ \forall^* x\in \MPT\ (g\cdot x\in P'). $$
Here $\forall^* x$ means ``for a comeager set of $x$." This implies
$$ \forall^* g\in \MPT\ \forall^* x\in\MPT\ (g\cdot x\in P'). $$
By Kuratowski--Ulam, we can switch the category quantifiers and get
$$ \forall^* x\in\MPT\ \forall^*g\in\MPT\  (g\cdot x\in P'). $$
This means that the set $C=\{x\in \MPT\,|\, \forall^*g\in\MPT\ (g\cdot x\in P')\}$ is comeager. However, $C$ is an invariant $G_\delta$ subset of $P$. Now $\pi^{-1}(C)$ is an invariant dense $G_\delta$ set in $X$ since $X$ is a model. Thus $\pi^{-1}(P)$ is an invariant comeager set in $X$, hence a generic dynamical property.

Conversely, let $A\subseteq X$ be an invariant comeager subset. We show that the saturation of $\pi(A)$ is an invariant comeager set in $\MPT$. If the saturation of $\pi(A)$ is not comeager, then its complement $C$ is nonmeager. Let $U\subseteq \MPT$ be an open set in which $C$ is comeager. By a similar argument as the application of Kuratowski--Ulam above, we obtain a comeager set of $x$ in $U$ so that $\forall^* g\in\MPT\ (g\cdot x\in P')$ for some $G_\delta$ set $P'$. But the set of $x$ so that $\forall^* g\in\MPT\ (g\cdot x\in P')$ is an invariant dense $G_\delta$ (the density follows from the consequence of Rokhlin's lemma that every orbit is dense). This says that $C$ is comeager. Thus by the above argument again $\pi^{-1}(C)$ is invariant comeager. Thus $\pi^{-1}(C)\cap A\neq\emptyset$, a contradiction.
\end{proof}

\section{The constructive symbolic definition}
In this section we consider the constructive symbolic definition of a rank one measure preserving system.

\begin{definition} Let $\mathcal{F}$ be the set of all finite words over the alphabet $\{0,1\}$ that start and end with $0$.
\begin{enumerate}
\item[(a)] A {\em generating sequence} is an infinite sequence $(v_n)$ of finite words in $\mathcal{F}$ defined by induction on $n\in\mathbb{N}$:
$$ v_0=0,\ \ v_{n+1}=v_n1^{a_{n,1}}v_n1^{a_{n,2}}\dots v_n1^{a_{n,q_n-1}}v_n $$
for some positive integer $q_n\geq 2$ and nonnegative integers $a_{n,i}$ for $1\leq i\leq q_n-1$. We continue to refer to the sequence $(q_n)$ as the {\em cutting parameter} and the doubly-indexed sequence $(a_{n,i})$ as the {\em spacer parameter}.
\item[(b)] An infinite word $V\in\{0,1\}^\mathbb{N}$ is a {\em rank one word} if there is a generating sequence $(v_n)$ such that $V\upharpoonright \mbox{lh}(v_n)=v_n$ for all $n\in\mathbb{N}$.
\item[(c)] A {\em rank one (topological dynamical) system} is
$$\begin{array}{r} X=\{ x\in\{0,1\}^\mathbb{Z}\mid \mbox{every finite subword of $x$ is a subword of $v_n$}\\
\mbox{\ \ \ \ \ \ \ \ \ \ \ \ \ \ \ \ \ \ \ \ \ \ \ \ \ \ \ \ \ \ \ \ \ \ \ \ \ \ \ \ \ \ \ \ \ \ \ \ \ \ \ \ \ \ \ \ \ \ \ \ \ \ \ \ \ \ \ \ \ \ \ for some $n\in\mathbb{N}$}\},\end{array} $$
where $(v_n)$ is a generating sequence, with the shift map $\sigma: X\to X$ defined by
$$ \sigma(x)(k)=x(k+1) \mbox{ for all $k\in\mathbb{Z}$.} $$
\end{enumerate}
\end{definition}
Every rank one topological dynamical system $X$ is a {\em Bernoulli subshift}, as $X$ is a shift-invariant closed subset of $\{0,1\}^\mathbb{Z}$. If $X$ is a rank one system, $\alpha$ is a finite word and $k\in\mathbb{Z}$, then
$$ U_{\alpha,k}=\{ x\in X\,\mid\, \mbox{ $x$ has an occurrence of $\alpha$ (starting) at position $k$}\} $$
is a basic open set of $X$. We consider only {\em nondegenerate} rank one systems; these are the rank one systems where the infinite rank one word $V$ given by the generating sequence is aperiodic. For any nondegenerate rank one system $X$ there is an atomless shift-invariant (possibly infinite) measure $\mu_0$ on $X$ defined by
$$ \mu_0(U_{\alpha,k})=\lim_{n\to\infty}\displaystyle\frac{\mbox{the number of occurrences of $\alpha$ in $v_n$}}{\mbox{the number of occurrences of $0$ in $v_n$}}. $$
$\mu_0$ is the unique shift-invariant measure with $\mu_0(U_{0,0})=1$. If $\mu_0$ is finite (equivalently, $\mu_0(U_{1,0})$ is finite), then its normalization $\mu$ is given by
$$ \mu(U_{\alpha,k})=\lim_{n\to\infty}\displaystyle\frac{\mbox{the number of occurrences of $\alpha$ in $v_n$}}{\mbox{lh}(v_n)}. $$
In this case, $\mu$ is the unique shift-invariant, atomless, probability Borel measure on $X$ (and is therefore ergodic). We summarize this in the following definition of symbolic rank one measure preserving systems.

\begin{definition}\label{rankonesymbolic} A {\em symbolic rank one (measure preserving) system} is a quadruple $(X,\mathcal{B},\mu,\sigma)$ such that $(X,\sigma)$ is a nondegenerate rank one topological dynamical system given by a generating sequence $(v_n)$, and $\mu$ is the unique shift-invariant, atomless, probability Borel measure on $X$, provided that
$$ \lim_{n\to\infty}\displaystyle\frac{\mbox{the number of occurrences of $1$ in $v_n$}}{\mbox{the number of occurrences of $0$ in $v_n$}}<+\infty. $$
\end{definition}

This definition has several advantages over the other definitions of rank one. First, with the symbolic definition the cutting and stacking process is represented by inserting 1s in between different copies of finite words, and is somewhat easier to visualize. It is clear that the combinatorics of the generating sequences completely determines the topological and measure structure of the rank one systems, so the study of rank one systems are very often reduced to a combinatorial analysis of rank one words. Additionally, one can study the rank one topological dynamical systems in their own right. This was exactly what the authors did in \cite{GaoHill1}. We will see in the rest of this section that some results obtained for the simpler topological setting will become very relevant even when we consider problems in the measure-theoretic context.

Our main objective of this paper is to establish a model for rank one measure preserving systems by exploring the symbolic definition of rank one. Instead of considering the rank one systems themselves, the elements of the model will be ``codes" for the rank one systems, and in this context the codes will be infinite rank one words.

Following \cite{GaoHill1} we let $\mathcal{R}$ be the set of all nondegenerate infinite rank one words. If $V\in \mathcal{R}$, we denote
$$ X_V=\{x\in \{0,1\}^\mathbb{Z}\,\mid\, \mbox{every finite subword of $x$ is a subword of $V$}\}. $$
If $(v_n)$ is a generating sequence for $V$, then $X_V$ coincides with the rank one system given by $(v_n)$, since every finite subword of $V$ is also a subword of some $v_n$. This means that $X_V$ is independent of the choice of the generating sequence for $V$, and is the unique rank one system associated with $V$. Moreover, it has been proved in \cite{GaoHill1} (Proposition 2.36) that if $V, W\in\mathcal{R}$ are distinct, then $X_V\neq X_W$. Thus there is a one-one correspondence between symbolic rank one systems and nondegenerate infinite rank one words, and it makes sense to think of nondegenerate infinite rank one words as codes for symbolic rank one systems.

The following notation will be used in the rest of the paper. If $\alpha$ is a finite word, we let
$$ Y(\alpha)=\mbox{ the number of occurrences of 1 in $\alpha$}, $$
$$ Z(\alpha)=\mbox{ the number of occurrences of 0 in $\alpha$}, $$
and $$ \rho_\alpha=\displaystyle\frac{Y(\alpha)}{Z(\alpha)}. $$
Clearly $Y(\alpha)+Z(\alpha)=\mbox{lh}(\alpha)$.

Now if $V\in \mathcal{R}$ and $(v_n)$ is a generating sequence for $V$, we claim that
$\lim_{n\to\infty}\rho_{v_n}$
either exists or is $+\infty$, and is independent of the choice of $(v_n)$. To see this we will need to recall some more concepts and results from \cite{GaoHill1}, including the following key notion of {\em canonical generating sequence} for rank one systems.

\begin{definition} \label{defcgs} Let $\mathcal{F}$ be the set of all finite words over the alphabet $\{0,1\}$ that start and end with $0$.
\begin{enumerate}
\item[(a)] If $u, v\in \mathcal{F}$, we say that $u$ is {\em built from} $v$ if there is a positive integer $q\geq 2$ and nonnegative integers $a_1,\dots, a_{q-1}$ such that
$$ u=v1^{a_1}v\dots v1^{a_{q-1}}v. $$
Moreover, we say that $u$ is {\em built simply from} $v$ if $a_1=\dots=a_{q-1}$.
\item[(b)] If $V$ is an infinite rank one word and $v\in\mathcal{F}$, we say that $V$ is {\em built from} $v$ if
there are nonnegative integers $a_1,\dots, a_n, \dots$ such that
$$ V=v1^{a_1}v\dots v1^{a_n}v\dots. $$
\item[(c)] If $V$ is an infinite rank one word and $v\in\mathcal{F}$, then $v$ is an element of the {\em canonical generating sequence} of $V$ if $V$ is built from $v$ and there are no $u, w\in\mathcal{F}$ such that
\begin{itemize}
\item[(i)] $\mbox{lh}(u)<\mbox{lh}(v)<\mbox{lh}(w)$,
\item[(ii)] $V$ is built both from $u$ and from $w$,
\item[(iii)] $w$ is built from $u$ and $u$ is built from $v$, and
\item[(iv)] $w$ is built simply from $u$.
\end{itemize}
\end{enumerate}
\end{definition}

The following basic fact is easy to see.

\begin{lemma}\label{rhoincrease}If $\alpha,\beta\in\mathcal{F}$ and $\alpha$ is built from $\beta$, with the cutting parameter $q$ and spacer parameter summing up to $a$, i.e.,
$$ \alpha=\beta1^{a_1}\beta\dots\beta1^{a_{q-1}}\beta $$
where $a=\sum^{q-1}_{k=1}a_k$, then
$$ \rho_\alpha=\rho_\beta+\displaystyle\frac{a}{qZ(\beta)}.
$$
\end{lemma}
\begin{proof} By a straightforward computation.\end{proof}

From this it follows that
for any generating sequence $(v_n)$ for $V\in\mathcal{R}$, the sequence
$\rho_{v_n}$ is nondecreasing, and hence $\lim_{n\to\infty}\rho_{v_n}$ either exists or is $+\infty$.

As an aside we also note the following basic fact which we will use later our proofs.

\begin{lemma}[\cite{GaoHill1} Lemma 2.7 (b)]\label{simplybuilt} Suppose $\alpha,\beta, \gamma\in\mathcal{F}$, $\alpha$ is built from $\beta$, and $\beta$ is built from $\gamma$. If $\alpha$ is built simply from $\gamma$, then $\beta$ is built simply from $\gamma$ and $\alpha$ is built simply from $\beta$.
\end{lemma}

Regarding the notion of canonical generating sequence, we remark that, although it is not clear from the definition, the canonical generating sequence is indeed a generating sequence. We will need the following results established in \cite{GaoHill1}.

\begin{theorem}[\cite{GaoHill1} Proposition 2.15]\label{cgs} Every nondegenerate infinite rank one word has a unique infinite canonical generating sequence.
\end{theorem}

\begin{lemma}[\cite{GaoHill1} Proposition 2.13]\label{cgsprop} If $v\in\mathcal{F}$ is an element of the canonical generating sequence of $V\in\mathcal{R}$ and $V$ is built from $u\in\mathcal{F}$, then
\begin{itemize}
\item if $\mbox{lh}(v)<\mbox{lh}(u)$ then $u$ is built from $v$;
\item if $\mbox{lh}(v)>\mbox{lh}(u)$ then $v$ is built from $u$;
\item if $\mbox{lh}(v)=\mbox{lh}(u)$ then $v=u$.
\end{itemize}
\end{lemma}

\begin{lemma}[\cite{GaoHill1} Proposition 2.17]\label{common}
Let $V, W\in\mathcal{R}$ and $(v_n)$ and $(w_m)$ be the canonical generating sequences for $V$ and $W$, respectively. Suppose for some $n \in\mathbb{N}$, $W$ is built from $v_{n+1}$. Then $v_i=w_i$, for all $i\leq n$.
\end{lemma}

Note that if $v\in \mathcal{F}$ is an element of any generating sequence for $V\in\mathcal{R}$, then $V$ is built from $v$. From this and Lemmas~\ref{rhoincrease} and \ref{cgsprop} it follows that if $(v_n)$ is the canonical generating sequence for $V\in \mathcal{R}$ and $(u_m)$ is any generating sequence for $V$, then
$\lim_{n\to\infty} \rho_{v_n}=\lim_{m\to\infty} \rho_{u_m}$.
In particular, this limit is independent of the choice of the generating sequence for $V$. For future reference we denote this limit as $\rho_V$. Definition~\ref{rankonesymbolic} states that rank one systems correspond to precisely those $V\in\mathcal{R}$ with $\rho_V<+\infty$.

We let
$$ \mathcal{R}^*=\left\{\,V\in\mathcal{R}\,\mid\, \rho_V<+\infty\,\right\}. $$
If $V\in\mathcal{R}^*$, then there is a unique shift-invariant, atomless, Borel probability on $X_V$, which we denote by $\mu_V$. Now $\mathcal{R}^*$ can be viewed as the space of codes for rank one measure preserving systems.

There is, however, a minor problem we have not mentioned so far: the constructive symbolic definition and the constructive geometric definition of rank one are not quite equivalent. In fact, in the case of an odometer map the spacer parameter is a constant $0$ sequence and the corresponding infinite rank one word is degenerate. Thus odometer maps are not obviously coded by elements of $\mathcal{R}^*$. Nevertheless, we will prove elsewhere that every odometer map is isomorphic to a nondegenerate rank one system. Thererfore every element of $\Rone$ does correspond to a nondegenerate rank one word, i.e., an element of $\mathcal{R}^*$.

We are now ready to discuss the topologies on $\mathcal{R}$ and $\mathcal{R}^*$.

\section{A topology on $\mathcal{R}^*$}
The notion of the canonical generating sequence enables us to define a rather natural topology on $\mathcal{R}$ given by the following metric.

\begin{definition} If $V, W\in\mathcal{R}$ are distinct infinite rank one words and $(v_n)$ and $(w_m)$ are canonical generating sequences for $V$ and $W$, respectively, then let
$$ d(V,W)=2^{-\sup\{\mbox{\scriptsize lh}(v)\,\mid\, v\in\{v_n\,\mid\, n\in\mathbb{N}\}\cap\{w_m\,\mid\, m\in\mathbb{N}\}\}}. $$
\end{definition}

It was shown in \cite{GaoHill1} (Proposition 2.20) that $d$ is a separable, complete ultrametric on $\mathcal{R}$. In particular, the metric topology on $\mathcal{R}$ given by $d$, which we denote by $\tau_d$, is Polish.


We are now ready to define a new topology on $\mathcal{R}^*$. This is achieved by the following two definitions.

\begin{definition}
\begin{enumerate}
\item[\rm (a)] For all $V\in \mathcal{R}$ and $n\in\mathbb{N}$, define
$\rho_{V,n}=\rho_{v_n},$
where $(v_n)$ is the canonical generating sequence of $V$.
\item[\rm (b)] For all $N\in\mathbb{N}$ and positive $r\in\mathbb{Q}$, define $\mathcal{O}(N,r)$ as the set of all $V\in \mathcal{R}$ such that for all $m,n\geq N$, $|\rho_{V, m}-\rho_{V,n}|\leq r$.
\end{enumerate}
\end{definition}

Clearly $\rho_V=\lim_{n\to\infty}\rho_{V,n}$. It is straightforward to check that each $\mathcal{O}(N,r)$ is a $\tau_d$-closed subset of $\mathcal{R}$, and for any positive $r\in\mathbb{Q}$, 
$$\mathcal{R}^*=\bigcup_{N\in\mathbb{N}}\mathcal{O}(N,r). $$

\begin{definition} Let $\tau$ be the topology on $\mathcal{R}$ generated by $$\tau_d\cup\{\mathcal{O}(N,r)\,|\, N\in\mathbb{N}, r\in\mathbb{Q}, r>0\}.$$
Let $\tau^*=\tau\!\upharpoonright\!\mathcal{R}^*$.
\end{definition}

By some general facts known as the change of topology techniques (cf. Lemmas 13.2 and 13.3 of \cite{Kechris}), the topology $\tau$ on $\mathcal{R}$ is Polish. In fact, each $\mathcal{O}(N,r)$ becomes a $\tau$-clopen set, and therefore $\mathcal{R}^*$ becomes a $\tau$-open subset of $\mathcal{R}$. This shows that $\tau^*$ is Polish. We note the following property of $\tau^*$.

\begin{lemma}\label{tau*} Let $V_k, V\in \mathcal{R}^*$ for all $k\in\mathbb{N}$. If $V_k\to V$ in $\tau^*$, then $d(V_k, V)\to 0$ and $\rho_{V_k}\to \rho_V$ as $k\to\infty$.
\end{lemma}

\begin{proof}
Suppose $V_k\to V$ in $\tau^*$. Since $\tau^*$ is finer than $\tau_d$, it follows that $d(V_k,V)\to 0$. To see   that $\rho_{V_k}\to \rho_V$, let $\epsilon>0$.  Let $r<\epsilon$ be a positive rational. Since $V\in \mathcal{R}^*$ there is $N$ such that $V\in \mathcal{O}(N,r)$. Then for all $n \geq N$, $\rho_{V, n} - \rho_{V,N} \leq r < \epsilon$.  Since $\rho_V$ is the limit, as $n \to \infty$, of $\rho_{V, n}$ we know $\rho_{V} - \rho_{V,N} \leq r < \epsilon$, which implies $\rho_{V, N} > \rho_V - \epsilon$.  Since $\mathcal{O}(N,r)\cap\mathcal{R}^*$ is $\tau^*$-open, there is $K_0$ such that for all $k>K_0$, $V_k\in \mathcal{O}(N,r)$.  Let $v_N$ be the $N$-th element of the canonical generating sequence of $V$. Let $K_1\geq K_0$ be such that for all $k>K_1$, $d(V_k,V)<2^{-\mbox{\scriptsize lh}(v_N)}$.
Now we claim that for all $k>K_1$, $|\rho_{V_k}-\rho_V|<\epsilon$, which completes the proof.

Fix any $k>K_1$. Let $v_M\in\mathcal{F}$ be the $M$-th element of the canonical generating sequence for $V$ that is also the longest common element of the canonical generating sequences for $V_k$ and $V$. Then $d(V_k,V)=2^{-\mbox{\scriptsize lh}(v_M)}<2^{-\mbox{\scriptsize lh}(v_N)}$. It follows that $M>N$. By Lemma~\ref{common}, the first $M$ many elements of the canonical generating sequence of $V_k$ coincide with the first $M$ many elements of the canonical generating sequence of $V$. Thus $\rho_{V_k,i}=\rho_{V,i}$ for all $i\leq M$.

Since $V_k\in\mathcal{O}(N,r)\subseteq \mathcal{O}(M,r)$, we have that for all $m\geq M$, $\rho_{V_k, m}-\rho_{V_k, M}\leq r<\epsilon$. Letting $m\to\infty$, we obtain $\rho_{V_k}-\rho_{V, M}<\epsilon$. This implies
$$\rho_{V_k}<\rho_{V,M}+\epsilon\leq \rho_V+\epsilon.$$
On the other hand, $$\rho_{V_k}\geq \rho_{V_k, M}\geq \rho_{V_k,N}=\rho_{V,N}>\rho_V-\epsilon.$$
This completes the proof.
\end{proof}

\section{The main theorem}

We are going to show that $\mathcal{R}^*$, equipped with $\tau^*$, is a model for all rank one measure preserving transformations in the sense of Definition~\ref{defmodel}. Before stating the result rigorously, we need to specify the translation map from codes in $\mathcal{R}^*$ to transformations in $\Rone$.

For this purpose fix $V\in \mathcal{R}^*$. The rank one system coded by $V$ is $(X_V,\mu_V,\sigma)$. Let $(v_n)$ be the canonical generating sequence for $V$, and let $(q_n)$ and $(a_{n,i})$ be the cutting parameter and spacer parameter, respectively, that are given by $(v_n)$. Let $a=\mu_V(U_{0,0})=1/(1+\rho_V)$. Then $0<a<1$. To define a rank one transformation $T\in \Rone$, we start with $F_0=[0,a]$ and then follow the cutting and stacking process determined by the cutting parameter $(q_n)$ and spacer parameter $(a_{n,i})$. More specifically, assume $F_n$ has been defined as an interval of the form $[0,\alpha]$ and $\bigcup^{h_n-1}_{r=0}T^r(F_n)$ is an interval of the form $[0,\beta]$ with $\beta\geq \alpha$. Then define $\{F_{n,i}\,|\, 1\leq i\leq q_n\}$ to be a partition of $F_n$ into intervals of equal length $\alpha/q_n$, and $\{C_{n,i,j}\,|\, 1\leq i< q_n, 1\leq j\leq a_{n,i}\}$ be a collection of disjoint intervals of length $\alpha/q_n$, with
$$\bigcup^{q_n-1}_{i=1}\bigcup^{a_{n,i}}_{j=1} C_{n,i,j}=\left(\beta, \beta+\sum^{q_n-1}_{i=1}a_{n,i} \displaystyle\frac{\alpha}{q_n}\right]. $$
According to Definition~\ref{defgeo} we must define $F_{n+1}=F_{n,1}=[0,\alpha/q_n]$ and the definitions of $T^r(F_{n+1})$ for all $1\leq r\leq h_{n+1}-1$ are also uniquely determined. At the end of this stage of the definition, $T$ is well defined except for the top level of the tower, formally $T^{h_{n+1}-1}(F_{n+1})$, and for the future spacers, formally $ (\beta+\alpha/q_n\sum^{q_n-1}_{i=1}a_{n,i} ,1]$. This finishes the inductive step of the definition. Note that this construction corresponds exactly to the way $v_{n+1}$ is built from $v_n$, with the new spacers correspondent to new $1$s inserted between diffferent copies of $v_n$.
Note that all the spacers eventually form a partition of the interval $(a,1]$, and the transformation $T$ is Lebesgue measure preserving. It is straightforward (albeit tedious) to check that $([0,1], T)\cong (X_V,\sigma)$. This translation map from $V$ to $T$ is denoted by $\pi: \mathcal{R}^*\to\Rone$.

The following is the main theorem of this paper.

\begin{theorem} $(\mathcal{R}^*,\pi)$ is a model for all rank one measure preserving transformations. That is, $\pi: \mathcal{R}^*\to \Rone$ is a continuous function such that for a comeager set $C\subseteq \Rone$ and for all $T\in C$, $\{V\in \mathcal{R}^*\,|\, \pi(V)\cong T\}$ is dense in $\mathcal{R}^*$. Moreover, $\pi$ is isomorphism-preserving.
\end{theorem}

As remarked earlier, this gives the desired corollary by Lemma~\ref{lemma12}.

\begin{corollary} The spaces $(\mathcal{R}^*,\tau^*)$ and $\Rone$ (with the weak topology) have the same generic dynamical properties.
\end{corollary}

The rest of the section is devoted to a proof of the main theorem.

That $\pi$ is isomorphism-preserving has been noted above. We need to show that $\pi$ is continuous and that it satisfies a density property in the definition of a model.

We first verify the continuity of $\pi$. For this, let $V_k\to V$ in $\tau^*$. Denote $T_k=\pi(V_k)$ and $T=\pi(V)$. We will show that $T_k\to T$ weakly, i.e., for all Borel $A\subseteq [0,1]$,
$\lambda(T_k(A)\triangle T(A))\to 0$. Let $\epsilon>0$ and $A\subseteq [0,1]$ be an arbitrary Borel set. We need to find $K\in\mathbb{N}$ such that for all $k\geq K$, $\lambda(T_k(A)\triangle T(A))<\epsilon$.

Let $\mathcal{A}$ be the $\sigma$-algebra generated by all sets $A$ of the form
$\bigcup_{s\in S} T^s(F_n)$, where $n\in\mathbb{N}$, $S\subseteq \{0,1,\dots, h_n-1\}$ and $F_n$ as in the definition of $T$ (i.e., $A$ is a union of some levels of the Rokhlin tower constructed up to a certain stage). Then $\mathcal{A}$ is dense as a measure subalgebra of all Borel sets. Thus in our proof of the continuity of $\pi$ we may assume that $A\in\mathcal{A}$.

Let $(v_n)$ be the canonical generating sequence for $V$. Fix $n_0$ and $S_0\subseteq \{0,1,\dots, h_{n_0}-1\}$ so that $A=\bigcup_{s\in S_0}T^s(F_{n_0})$. Note that for any $n\geq n_0$, $A=\bigcup_{s\in S} T^s(F_n)$ for an appropriate $S\subseteq \{0,1,\dots h_n-1\}$. Therefore, without loss of generality, we may assume that $\lambda(F_{n_0})<\epsilon/4$.

Now consider the cutting and stacking process we have followed to obtain the Rokhlin tower corresponding to $v_{n_0}$ as described above. If this process is applied to $F_{x,0}=[0,x]$, with $0<x<1$ a variable, we would obtain a different Rokhlin tower as $x$ varies. Let $F_{x,n_0}$ be the first level of the resulting tower, $T_x$ be the partial transformation defined in this process, and $$A_x=\bigcup_{s\in S_0}T_x^s(F_{x,n_0}).$$
Note that that $T=T_a$ and $A=A_a$, where $a=\mu_V(U_{0,0})=1/(1+\rho_V)$. Define
$$ f(x)=\lambda(A_x\triangle A)=\lambda(A_x\triangle A_a). $$
Then $f(x)$ is continuous at $x=a$.
Furthermore, let
$$ B_x=\bigcup_{s\in S_0-\{h_{n_0}-1\}} T^s_x(F_{x,n_0})=A_x- T_x^{h_{n_0}-1}(F_{x,n_0}). $$
Then $T_x$ is well defined on $B_x$. Define
$$ g(x)=\lambda(T_{x}(B_x)\triangle T_a(B_a))=\lambda(T_x(B_x)\triangle T(B_a)). $$
Then $g(x)$ is also continuous at $x=a$.

By Lemma~\ref{tau*} we have $d(V_k,V)\to 0$. It follows that there is $K_0\in\mathbb{N}$ such that for all $k\geq K_0$, $v_{n_0}$ is an element of the canonical generating sequence for $V_k$. Thus the construction of $T_k$, for $k\geq K_0$, follows the cutting and stacking process corresponding to $v_{n_0}$ applied to $x_k=1/(1+\rho_{V_k})$. By Lemma~\ref{tau*} again $\rho_{V_k}\to\rho_V$ and therefore $x_k\to a$.
Let $A_k=A_{x_k}$ and $B_k=B_{x_k}$. Since $x_k\to a$ and $\lambda(F_{n_0})=\lambda(F_{a, n_0})<\epsilon/4$, there is $K_1\geq K_0$ such that for all $k>K_1$, $\lambda(F_{x_k,n_0})<\epsilon/4$. Thus for all $k\geq K_1$, $\lambda(A_k\triangle B_k)<\epsilon/4$. We also have $\lambda(A\triangle B_a)<\epsilon/4$. It follows from the continuity of $f$ that there is $K_2\geq K_1$ such that for all $k\geq K_2$,
$$ \lambda(A_k\triangle A)<\displaystyle\frac{\epsilon}{4}. $$
Similarly, it follows from the continuity of $g$ that there is $K\geq K_2$ such that for all $k\geq K$,
$$ \lambda(T_k(B_k)\triangle T(B_a))<\displaystyle\frac{\epsilon}{4}. $$
Thus for all $k\geq K$,
$$\begin{array}{rcl} & & \lambda(T_k(A)\triangle T(A))\\ \\
&\leq &\lambda(T_k(A)\triangle T_k(A_k))+\lambda(T_k(A_k)\triangle T_k(B_k)) +\lambda(T_k(B_k)\triangle T(B_a))\\
& &+\lambda(T(B_a)\triangle T(A)) \\ \\
&\leq & \lambda(A_k\triangle A)+\lambda(A_k\triangle B_k)+\lambda(T_k(B_k)\triangle T(B_a))+\lambda(B_a\triangle A) \\ \\
&\leq & \epsilon/4+\epsilon/4+\epsilon/4+\epsilon/4=\epsilon.
\end{array}
$$
This finishes the proof that $\pi$ is continuous.

To complete the proof of the main theorem, it remains to verify that for a comeager set $C\subseteq\Rone$ and for all $T\in C$, $\{V\in\mathcal{R}^*\,|\, \pi(V)\cong T\}$ is dense in $\mathcal{R}^*$. Our comeager set $C$ is the set of all rank one transformations that are not (isomorphic to) odometer maps. By the constructive geometric definition, each rank one transformation $T$ corresponds to some generating sequence $(v_n)$, and therefore to the infinite rank one word $V$ that is the limit of $v_n$. If $V$ is periodic, then $T$ is isomorphic to an odometer map. Thus if $T\in C$, then the infinite rank one word $V$ corresponding to it must be nondegenerate, hence $V\in\mathcal{R}^*$. This shows that if $T\in C$, then the set $\{V\in\mathcal{R}^*\,|\, \pi(V)\cong T\}$ is nonempty. To complete the proof it suffices to show that every measure-theoretic isomorphism equivalence class in $\mathcal{R}^*$ is dense. We will instead show a stronger statement that every topological isomorphism equivalence class in $\mathcal{R}^*$ is dense.

For this, we recall some concept and result from \cite{GaoHill1} about topological isomorphism in $\mathcal{R}$. If $V, W\in\mathcal{R}$ and $v, w\in\mathcal{F}$ with $\mbox{lh}(v)=\mbox{lh}(w)$, the pair $(v,w)$ is called a {\it replacement scheme} for $V$ and $W$ if $V$ is built from $v$ and $W$ is built from $w$, say
$$ V=v1^{a_0}v1^{a_1}v\dots \mbox{ and } W=w1^{b_0}w1^{b_1}w\dots, $$ and $a_i=b_i$ for all $i\in\mathbb{N}$. One of the main results of \cite{GaoHill1} is the following theorem characterizing the topological isomorphism between rank one systems.

\begin{theorem}[\cite{GaoHill1} Corollary 3.7]\label{topiso} For $V, W\in\mathcal{R}$, $(X_V,\sigma)$ and $(X_W,\sigma)$ are topologically isomorphic iff there is a replacement scheme $(v,w)$ for $V$ and $W$. Moreover, the $v$ and $w$ occurring in the replacement scheme may be required to be elements of the canonical generating sequences for $V$ and $W$, respectively.
\end{theorem}

For $V, W\in\mathcal{R}$, denote $V\approx W$ if there is a replacement scheme for $V$ and $W$. Then $\approx$ is an equivalence relation. Note that $\mathcal{R}^*$ is an $\approx$-invariant subset of $\mathcal{R}$.

We show that given $U, V\in \mathcal{R}^*$ and $\tau^*$-open set $\mathcal{U}$ with $U\in\mathcal{U}$, there is $W\in\mathcal{U}$ with $V\approx W$. Without loss of generality we may assume $\mathcal{U}$ is a basic $\tau^*$-open set of the form
$$ \{ U'\in\mathcal{R}^*\,|\, d(U', U)<\epsilon\}\cap \mathcal{O}(N_1,r_1)\cap\dots\cap\mathcal{O}(N_p,r_p) $$
for some $\epsilon>0$, $p\in\mathbb{N}$, $N_1,\dots, N_p\in\mathbb{N}$, and positive $r_1,\dots,r_p\in\mathbb{Q}$.

Let $(u_n)$ and $(v_m)$ be the canonical generating sequences for $U$ and $V$, respectively. Fix a large enough $n_0$ such that $n_0> N_1,\dots, N_p$ and $2^{-\mbox{\scriptsize lh}(u_{n_0})}<\epsilon$.
For notational simplicity we denote $u_{n_0}$ by $u$ in the rest of the proof.

Our plan is to find a large enough $m_0$ and, letting $v=v_{m_0}$, define a word $w\in\mathcal{F}$ by
$$ w=u^{k_0}1^{t_0}u $$
so that $\mbox{lh}(v)=\mbox{lh}(w)$ and $\mbox{lh}(u)\leq t_0< 2\mbox{lh}(u)$. Note that $k_0$ and $t_0$ are uniquely determined once we fix $\mbox{lh}(v)$. In fact $k_0=\lfloor \mbox{lh}(v)/\mbox{lh}(u)\rfloor-2$. We will make sure that $k_0\geq 2$. We then obtain an infinite rank one word $W$ by using $(v,w)$ as a replacement scheme. It will be the case that the first $n_0+1$ many elements of the canonical generating sequence for $W$ are $u_i$, $1\leq i<n_0$, $u$ and $w$. Thus we fullfil the requirements $V\approx W$ and $d(W,U)<\epsilon$. We will show that as long as we choose $m_0$ large enough, the resulting $W$ will be in $\mathcal{O}(N_1,r_1)\cap\dots\cap \mathcal{O}(N_p,r_p)$.

Let $M_0$ be the least such that for all $m\geq M_0$, $\mbox{lh}(v_m)\geq 4\mbox{lh}(u)$. For each $m\geq M_0$, consider the quantity
$$ c_m=\displaystyle\frac{Z(v_m)}{\lfloor \mbox{lh}(v_m)/\mbox{lh}(u)\rfloor-1}. $$
It is straightforward to see that the sequence $c_m$, $m\geq M_0$, is nonincreasing. Thus $c_m\leq c_{M_0}$ for all $m\geq M_0$. For notational simplicity we denote $c_{M_0}$ by $c$ in the rest of the proof.

For each $1\leq j\leq p$, let
$$ \delta_j=r_j-\rho_{U,n_0}+\rho_{U,N_j}. $$
For each $1\leq j\leq p$, since $U\in \mathcal{O}(N_j,r_j)$, we have $|\rho_{U,n_0}-\rho_{U,N_j}|\leq r_j$. Noting that $(\rho_{U,n})$ is a nondecreasing sequence (Lemma~\ref{rhoincrease}) that is not eventually constant ($U$ is nondegenerate), and $n_0$ was chosen to be greater than each $N_j$, we actually have
$$ 0\leq \rho_{U,n_0}-\rho_{U,N_j}<r_j. $$
Thus $\delta_j>0$. Let
$$ \delta=\mbox{min}\{\delta_1,\dots,\delta_p\}. $$
Let $M_1\geq M_0$ be the such that for all $m\geq M_1$,  
$$ \displaystyle\frac{2\mbox{lh}(u)}{\lfloor \mbox{lh}(v_m)/\mbox{lh}(u)\rfloor-1}<\displaystyle\frac{\delta}{2}. $$

Finally, let $r$ be a positive rational such that $2<\frac{\delta}{2c}$. Let $m_0\geq M_1$ be such that $V\in\mathcal{O}(m_0,r)$.

We prove that this $m_0$ works. Once we have fixed $m_0$, we have also defined $v$, $w$, $k_0$, $t_0$ and $W$ according to our plan specified above. Since $m_0\geq M_0$, $k_0\geq 2$. Define a sequence $(w_i)$ as follows. For $i\leq n_0$, let $w_i=u_i$. Define
$w_{n_0+1}=w$. For $i>n_0+1$, let $w_i$ be the word obtained using the replacement scheme $(v,w)$ from $v_{m_0+i-n_0-1}$. We claim that $(w_i)$ is the canonical generating sequence for $W$. First, $(w_i)$ is obviously a generating sequence for $W$. To verify the canonicity of the sequence, we repeatedly use Lemmas~\ref{simplybuilt} and \ref{cgsprop}. Note by our construction that $w$ is not built simply from $u$ and $w$ is not built from any word longer than $u$. This implies that both $u$ and $w$ are elements of the canonical sequence for $W$ by Definition~\ref{defcgs} (c) and the lemmas. Similar arguments apply also to $u_i$ for $1\leq i<n_0$. For $i>n_0+1$, $w_i$ is an element of the canonical generating sequence for $W$ because $v_{m_0+i-n_0-1}$ is an element of the canonical generating sequence for $V$.

So far we have verified that $V\approx W$ and $d(U,W)<\epsilon$. To finish the proof we fix a $1\leq j\leq p$ and show that $W\in \mathcal{O}(N_j, r_j)$. For this it suffices to verify that for all $i>N_j$, $\rho_{w_i}-\rho_{w_{N_j}}\leq r_j$. If $i\leq n_0$ this is obvious since $w_i=u_i$ for all $i\leq n_0$. The first nontrivial verification is for $i=n_0+1$. In this case we need to see that $\rho_w-\rho_{u_{N_j}}\leq r_j$. We have
$$\begin{array}{rcl} \rho_w-\rho_{u_{N_j}}&=&\rho_w-\rho_u+\rho_u-\rho_{u_{N_j}} \\ \\
&=& \rho_w-\rho_u+\rho_{U,n_0}-\rho_{U,N_j} \\ \\
&= & \displaystyle\frac{t_0}{(k_0+1)Z(u)}+r_j-\delta_j \\ \\
&\leq&r_j-\delta+\displaystyle\frac{2\mbox{lh}(u)}{\lfloor \mbox{lh}(v)/\mbox{lh}(u)\rfloor-1 } \\ \\
&< & r_j-\delta+\delta/2\ < r_j.
\end{array}
$$
This verifies the desired inequality for $\rho_w$. In addition, it also shows that $r_j-\rho_w+\rho_{u_{N_j}}>\delta/2$. Thus for $i>n_0+1$, we only need to verify that $\rho_{w_i}-\rho_w<\delta/2$.

Note that
$$ \displaystyle\frac{Z(v)}{Z(w)}=\displaystyle\frac{Z(v)}{(\lfloor \mbox{lh}(v)/\mbox{lh}(u)\rfloor-1)\cdot Z(u)}=\displaystyle\frac{c_{m_0}}{Z(u)}\leq c. $$
By Lemma~\ref{rhoincrease} we have that for $i>n_0+1$,
$$ \rho_{w_i}-\rho_w=\displaystyle\frac{1}{Z(w)}\cdot \displaystyle\frac{a}{q} $$
for some $a$ and $q$. Since $v_{m_0+i-n_0-1}$ is obtained by a replacement scheme $(w, v)$ from $w_i$,
we again use Lemma~\ref{rhoincrease} to get
$$ \rho_{v_{m_0+i-n_0-1}}-\rho_v=\displaystyle\frac{1}{Z(v)}\cdot \displaystyle\frac{a}{q} $$
with the same parameters $a$ and $q$ as above. However, $V\in \mathcal{O}(m_0, r)$. Thus
$$ \rho_{v_{m_0+i-n_0-1}}-\rho_v\leq r. $$
We therefore obtain
$$\begin{array}{rcl} \rho_{w_i}-\rho_w&=&\displaystyle\frac{1}{Z(w)}\cdot \displaystyle\frac{a}{q} \\ \\
&=&\displaystyle\frac{Z(v)}{Z(w)}\cdot\displaystyle\frac{1}{Z(v)}\cdot \displaystyle\frac{a}{q} \\ \\
&\leq & cr < c\displaystyle\frac{\delta}{2c}=\displaystyle\frac{\delta}{2}
\end{array}
$$
as desired. This finishes the proof of the main theorem.

\section{Another look at the topology on $\mathcal{R}^*$}

As we have seen in \cite{GaoHill1} and in earlier sections of this paper, the notion of the canonical generating sequence plays an essential role in our study of symbolic rank one systems. Even our definition of the topology $\tau^*$ on $\mathcal{R}^*$ relied on this notion. However, this reliance also makes the definition of $\tau^*$ somewhat technical and perhaps hard to use in potential applications. In this last section we give an alternate definition of the topology $\tau^*$ which does not mention the canonical generating sequences.

Recall that our definition of $\tau^*$ is based on a refinement of the Polish topology $\tau_d$ on $\mathcal{R}$, and the definition of $\tau_d$ already mentions the canonical generating sequences. Thus it is relevant to note the following result from \cite{GaoHill1}.

\begin{lemma}[\cite{GaoHill1} Proposition 2.19]\label{taud} The topology $\tau_d$ on $\mathcal{R}$ can be generated by basic open sets of the form
$$ \{ V\in\mathcal{R}\,|\, \mbox{$V$ is built from $v$}\} $$
where $v\in\mathcal{F}$.
\end{lemma}

This gives an alternate definition of $\tau_d$ without mentioning the canonical generating sequence. What we are doing below is to replace the basic open sets $\mathcal{O}(n,r)$ in the definition of $\tau^*$ by sets whose definitions do not mention the canonical generating sequences.

Suppose $\alpha, \beta\in\mathcal{F}$ and $\alpha$ is built from $\beta$, say
$$ \alpha=\beta 1^{a_1}\beta\cdots \beta 1^{a_{q-1}}\beta, $$
with cutting parameter $q$ and spacer parameter $a_i$, $1\leq i<q$. Define
$$ L(\alpha,\beta)=\displaystyle\frac{1}{\mbox{lh}(\alpha)}\displaystyle\sum_{j=1}^{q-1}a_j. $$
Now suppose $V$ is an infinite rank one word that is built from $v\in\mathcal{F}$. Let $(v_n)$ be any generating sequence of $V$ with $v=v_0$. Define
$$ L(V,v)=\lim_{n\to\infty} L(v_n,v). $$
The value $L(V,v)$ is well defined and does not depend on the choice of the generating sequence $(v_n)$.

To give $L(V,v)$ an intuitive interpretation it is helpful to recall the notion of {\em expectedness} which was discussed in  \cite{GaoHill1}. In the above expression of $\alpha$ in terms of $\beta$, there is a unique way to view $\alpha$ as a collection of disjoint occurrences of $\beta$ separated by 1s. We call the occurrences of $\beta$ in $\alpha$ that are demonstrated in the expression {\em expected}. In general there might be occurrences of $\beta$ in $\alpha$ other than those demonstrated explicitly in this expression, and these occurrences are called {\em unexpected}.
Similarly, if an infinite rank one word $V$ is built from $v\in\mathcal{F}$, then $V$ can be uniquely viewed as a collection of disjoint occurrences of $v$ separated by 1s.  We call elements of this collection {\em expected} and other occurrences of $v$ in $V$ {\em unexpected}. With this terminology, $L(V,v)$ can be viewed as the fraction of entries (necessarily $1$s) in $V$ that are not part of any expected occurrence of $v$ in $V$.

It not hard to see that $L(V,v)$ is related to $\rho_V$ and $\rho_v$ we defined before.  If $v_n$ is a generating sequence for $V$ with $v_0 = v$, then $L(v_n, v)$, the fraction of entries in $v_n$ that are not a part of any expected occurrence of $v$ can be calculated by $$\frac{\rho_{v_n} - \rho_v}{1 + \rho_{v_n}}.$$  Thus, $$L(V,v) = \frac{\rho_{V} - \rho_v}{1 + \rho_{V}} = 1 - \frac{1 + \rho_v}{1 + \rho_V}.$$
The following properties are also easy to check.
\begin{enumerate}
\item  If $V \notin \mathcal{R}^*$, then $L(V, v) = 1$.
\item  If $V \in \mathcal{R}^*$, then $0< L(V, v) <1$; moreover, if $(v_n)$ is any generating sequence for $V$, then $L(V, v_n)$ is a nonincreasing sequence whose limit is 0.
\end{enumerate}

We now define a topology $\tau^{**}$ on $\mathcal{R}^*$, which we will show is the same as $\tau^*$.

\begin{definition} For $v \in \mathcal{F}$ and $s \in \mathbb{Q}\cap (0,1)$, let $$\mathcal{U}(v, s) = \{V \in \mathcal{R}^*\,|\, \mbox{ $V$ is built from $v$ and }  L(V, v) \leq s\}.$$  Let $\tau^{**}$ be the topology on $\mathcal{R}^*$ generated by sets of the form $\mathcal{U}(v, s)$.
\end{definition}

Note that
$$\{V\in\mathcal{R}^*\,|\, \mbox{$V$ is built from $v$}\} = \bigcup_{s \in \mathbb{Q}\cap(0,1)}   \mathcal{U}(v,s). $$
Thus $\tau^{**}$ is a refinement of $\tau_d$ in view of Lemma~\ref{taud}.

\begin{proposition}
$\tau^* = \tau^{**}$.
\end{proposition}

\begin{proof}

We first show $\tau^{**}\subseteq \tau^*$. For this let $V \in \mathcal{U}(v, s)$.  We want to find a $\tau^*$-open set $\mathcal{O}$ with $V\in\mathcal{O}\subseteq \mathcal{U}(v,s)$.

Let $(v_n)$ be the canonical generating sequence for $V$ and let $N$ be as small as possible so that $\mbox{lh}(v_N)>\mbox{lh}(v)$.  By Lemma~\ref{cgsprop}, $v_N$ is built from $v$.  Now let $r = \frac{1 + \rho_v}{1 - s} - (1 + \rho_{v_N})$ and note that $r$ is positive and rational.  Let 
$$\mathcal{O} = \{W \in \mathcal{R}^*\,|\, d(V, W) < 2^{-\mbox{\scriptsize lh}(v)} \textnormal{ and } W \in \mathcal{O}(N, r)\}.$$

It is clear that $\mathcal{O}$ is $\tau^*$-open.  Note that each $W \in \mathcal{R}^*$ belongs to $\mathcal{O}$ iff $v_N$ is in the canonical generating sequence of $W$ and $$\rho_W \leq \frac{1 + \rho_v}{1-s} -1.$$  Simple algebra verifies (since $s<1$) that this last condition is equivalent to $1 - \frac{1 + \rho_v}{1+ \rho_W} \leq s$.  It is then straightforward to check that $V \in \mathcal{O} \subseteq \mathcal{U}(v, s)$.




For the other direction it suffices to show that $\mathcal{O}(N,r)$ is $\tau^{**}$-open for any $N \in\mathbb{N}$ and positive $r \in \mathbb{Q}$. Fix $N, r$ and let $V\in\mathcal{O}(N,r)$. We want to find $v\in\mathcal{F}$ and $s \in \mathbb{Q} \cap (0,1)$ such that $V\in \mathcal{U}(v,s)\subseteq \mathcal{O}(N,r)$. 

Let $(v_n)$ be the canonical generating sequence for $V$. Let $v=v_{N+1}$ and
$$ s = 1 - \frac{1+ \rho_v}{1 + (r + \rho_{v_N})}.$$  Since $V \in \mathcal{O}(N,r)$, we know that $\rho_V - \rho_{v_N} \leq r$ and thus, $\rho_V \leq r + \rho_{v_N}$.  This implies that $$1 - \frac{1 + \rho_v}{1 + \rho_V} \leq 1 - \frac{1 + \rho_v}{1 + (r + \rho_{v_N})} = s$$ and thus $V \in \mathcal{U}(v, s)$.  It remains to show that $\mathcal{U}(v, s) \subseteq \mathcal{O}(N,r)$.  

Let $W \in \mathcal{U}(v, s)$ and let $(w_m)$ be the canonical generating sequence for $W$. Since $W$ is built from $v=v_{N+1}$, Lemma~\ref{common} implies that $w_N=v_N$.  Also, we know $$1 - \frac{1 + \rho_v}{1 + \rho_W} \leq  s = 1 - \frac{1+ \rho_v}{1 + (r + \rho_{v_N})}.$$  This implies that $\rho_W \leq r - \rho_{v_N}$, which implies that $\rho_W - \rho_{v_N} \leq r$.  Since $v_N = w_N$, $W \in \mathcal{O}(N,r)$.

\end{proof}

It is possible to give a direct proof of the main theorem using the alternate definition of $\tau^*$, without using the notion of canonical generating sequences. However, the purpose of introducing the alternate definition is not to avoid the notion of canonical generating sequences in the study of symbolic rank one systems, but to make the main theorem easier to apply in the study of generic dynamical properties of rank one transformations.


\begin{thebibliography}{99}

\bibitem{Ferenczi}
{\sc S. Ferenczi},
Systems of finite rank,
{\it Colloq. Math.}
{73}:1
(1997), 35--65.

\bibitem{Foreman}
{\sc M. Foreman},
Models for measure preserving transformations,
{\it Topol. Appl.} 157 (2010), 1404--1414.

\bibitem{GaoHill1}
{\sc S. Gao and A. Hill},
Topological isomorphism for rank-1 systems,
{\it J. Anal. Math.}, to appear.

\bibitem{GlasnerKing}
{\sc E. Glasner and J. King},
A zero-one law for dynamical properties, {\it Topological Dynamics and Applications (Minneapolis, MN, 1995)},
231--242, Contemp. Math., 215, Amer. Math. Soc., Providence, RI, 1998.

\bibitem{Kechris}
{\sc A.S. Kechris},
{\it Classical Descriptive Set Theory}. Springer-Verlag, New York, 1995.

\end{thebibliography}
\end{document}